\nonstopmode
\documentclass[]{amsart}
\usepackage{verbatim}
\usepackage{amssymb,bbm}
\usepackage{graphicx}
\usepackage{psfrag}
\usepackage{ifpdf}
\usepackage{framed}
\usepackage{multirow}
\ifpdf
\usepackage[notightpage]{pst-pdf}
\fi

\allowdisplaybreaks
\usepackage[all]{xy}
\newdir{ >}{{}*!/-10pt/@{>}}
\newdir^{ (}{{}*!/-5pt/@^{(}}
\newdir_{ (}{{}*!/-5pt/@_{(}}
\def\cgaps#1{}
\def\Cgaps#1{}
\def\AMStag#1{}
\def\AMSunderset#1\to#2{\underset{#1}{#2}}
\def\AMSoverset#1\to#2{\overset{#1}{#2}}
\def\undersetbrace#1\to#2{\underbrace{#2}_{#1}}
\def\oversetbrace#1\to#2{\overbrace{#2}^{#1}}
\def\3{\ss}

\newcommand{\nmb}[2]{\ifx!#1{\ref{nmb:#2}}%
\else\if.#1{\label{nmb:#2}}%
\else\if0#1{\label{nmb:#2}}%
\else{{#2}}%
\fi\fi\fi}
\swapnumbers

\newtheorem*{proposition*}{Proposition}
\newtheorem{theorem}[subsection]{Theorem}
\newtheorem*{theorem*}{Theorem}
\newtheorem{lemma}[subsection]{Lemma}
\newtheorem*{lemma*}{Lemma}

\newtheorem*{corollary*}{Corollary}

\newtheorem*{result*}{Result}
\newtheorem*{remark}{Remark}

\def\cit#1#2{\ifx#1!\cite{2}\else#2\fi} 
\def\ign#1{}             
\def\o{\circ\,} 
\def\X{\mathfrak X} 
\def\al{\alpha} 
\def\be{\beta}

\def\ep{\varepsilon}

\def\la{\lambda} 
\def\rh{\rho} 
 
\def\ta{\tau} 
\def\ph{\varphi} 
 
\def\ps{\psi} 
 
\def\Ga{\Gamma}

\def\i{^{-1}}

\def\p{\partial}
\def\ad{\operatorname{ad}}

\def\F{\mathcal{F}}

\def\exp{\operatorname{exp}}
\let\on=\operatorname

\DeclareMathOperator{\Id}{Id}
\DeclareMathOperator{\Diff}{Diff}

\newcommand{\N}{\mathbb{N}}
\newcommand{\R}{\mathbb{R}}

\newcommand{\ud}{\,\mathrm{d}}

\newcommand{\one}{\mathbbm{1}}
\graphicspath{{images/}}
\newcommand{\executeiffilenewer}[3]{%
\ifnum\pdfstrcmp{\pdffilemoddate{#1}}%
{\pdffilemoddate{#2}}>0%
{\immediate\write18{#3}}\fi%
}

\title[Fractional order Sobolev 
metrics on the diffeomorphism group]
{Geodesic distance for right invariant Sobolev 
metrics of fractional order on the diffeomorphism group}
\author{Martin Bauer, Martins Bruveris, Philipp Harms, Peter W. Michor}
 
\address{
Martin Bauer, Peter W.\ Michor: Fakult\"at f\"ur Mathematik,
Universit\"at Wien, 
Nordbergstrasse 15, A-1090 Wien, Austria.\newline\indent
Martins Bruveris: Dep. of Mathematics, Imperial College, London SW7~2AZ, UK.
\newline\indent
Philipp Harms: Edlabs, Harvard University, 44 Brattle street, Cambridge, MA 02138
}
\email{bauer.martin@univie.ac.at}
\email{m.bruveris08@imperial.ac.uk}
\email{pharms@edlabs.harvard.edu}
\email{peter.michor@esi.ac.at }
\date{{\today} } 
\thanks{All authors were supported by `Fonds zur
F\"orderung der wissenschaftlichen                    
Forschung, Projekt P~21030'. Martin Bauer was partially supported by `Fonds zur
F\"orderung der wissenschaftlichen                    
Forschung, Projekt P~24625' and  Martins Bruveris was partially supported by the European Research Council.}
\keywords{diffeomorphism group, geodesic distance, Sobolev metrics of
non-integral order}
\subjclass[2010]{Primary 35Q31, 58B20, 58D05} 
\begin{document}
\begin{abstract}
We study Sobolev-type metrics of fractional order $s\geq0$ on the group $\Diff_c(M)$ of compactly supported diffeomorphisms of a manifold $M$. We show that for the important special case $M=S^1$ the geodesic distance on $\Diff_c(S^1)$ vanishes if and only if $s\leq\frac12$. For other manifolds we obtain a partial characterization: the geodesic distance on $\Diff_c(M)$ vanishes for $M=\R\times N, s<\frac12$ and for $M=S^1\times N, s\leq\frac12$, with $N$ being a compact Riemannian manifold. On the other hand the geodesic distance on $\Diff_c(M)$ is positive for $\dim(M)=1, s>\frac12$ and $\dim(M)\geq2, s\geq1$.

For $M=\R^n$ we discuss the geodesic equations for these metrics. For $n=1$ we obtain some well known PDEs of hydrodynamics: 
Burgers' equation for $s=0$, the modified Constantin-Lax-Majda equation for $s=\frac 12$ and the Camassa-Holm equation for $s=1$.
\end{abstract} 

\maketitle

\section{Introduction}
In the seminal paper \cite{Arnold1966} Arnold showed that the incompressible Euler equations can be seen as geodesic equations on the group of volume preserving diffeomorphisms with respect to the $L^2$-metric.
This interpretation was extended to other PDEs used in hydrodynamics, some of which are: Burgers' equation as the geodesic equation for $\on{Diff}_c(\R)$ with the $L^2$-metric, Camassa-Holm equation for $\on{Diff}_c(\R)$ with the $H^1$-metric or KdV for the Virasoro-Bott group with the $L^2$-metric in \cite{OvsienkoKhesin87}. Recently it was shown by Wunsch \cite{Wunsch2010a} that the modified Constantin-Lax-Majda equation (mCLM) is the geodesic equation on the homogeneous space $\on{Diff}(S^1)/S^1$ with respect to the homogeneous $\dot H^{1/2}$-metric.

The geometric interpretation was used by Ebin and Marsden in \cite{Ebin1970} to show the well-posedness of the Euler equations. These techniques were expanded and applied to other equations, see e.g. \cite{Constantin2007, Constantin2003, Wunsch2010, GayBalmaz2009}.

The interpretation of a PDE as the geodesic equation on an infinite dimensional manifold opens up a variety of geometrical questions, which may be asked about the manifold. What is the curvature of this manifold? Do geodesics have conjugate points? Do there exist totally geodesic submanifolds? The question we want to concentrate upon in this paper, that of geodesic distance, has a very simple answer in finite dimensions but shows more nuances in infinite dimensions. The geodesic distance between two points is defined as the infimum of the pathlength over all paths connecting the two points. In finite dimensions, because of the local invertibility of the exponential map, this distance is always positive and the topology of the resulting metric space is the same as the manifold topology. 

However, in infinite dimensions, this does not always hold: the induced geodesic distance for weak Riemannian metrics on infinite dimensional manifolds may vanish. This surprising fact was first noticed for the $L^2$-metric on shape space $\on{Imm}(S^1,\mathbb R^2)/\on{Diff}(S^1)$  in \cite[3.10]{Michor98}. Here $\on{Imm}(S^1,\mathbb R^2)/\on{Diff}(S^1)$ denotes the orbifold of all immersions $\on{Imm}(S^1,\mathbb R^2)$ of $S^1$ into $\R^2$ modulo reparametrizations. In \cite{Michor102} it was shown that this result holds for the general shape space $\on{Imm}(M,N)/\on{Diff}(M)$ for any compact manifold $M$ and Riemannian manifold $N$, and also for the right invariant $L^2$-metric (or equivalently Sobolev-type metric of order zero) on each full diffeomorphism group with compact support $\on{Diff}_c(N)$. In particular, since Burgers' equation is related to the geodesic equation of the right invariant $L^2$-metric on $\on{Diff}_c(\R^1)$, it implies that solutions of Burgers' equation are critical points of the 
length functional, but they are not length-minimizing. A similar result was shown for the KdV equation in \cite{Michor122}. 

On the other hand it was shown in \cite{Michor102,Michor119} that for Sobolev-type metrics on the diffeomorphism group of order one or higher the induced geodesic distance is positive. 
This naturally leads to the question whether one can determine the Sobolev order where this change of behavior occurs. This paper gives a complete answer in the case of $M=S^1$ and a partial answer for other manifolds. Our main result is:

\begin{theorem*}[Geodesic distance]
Let $M$ be a Riemannian manifold and 
$\Diff_c(M)$ the group of compactly supported diffeomorphisms of $M$.
\begin{enumerate}
\item
The geodesic distance
for the fractional order Sobolev type metric $H^s$ on $\Diff_c(M)$
vanishes for
\begin{itemize}
\item
$0\leq s < \frac 12$ and $M$ a Riemannian manifold that is the product of $\R$ with a compact manifold $N$, i.e., $M = \R\times N$.
\item
$s = \frac 12$ and  $M$ a Riemannian manifold that is the product of $S^1$ with a compact manifold $N$, i.e., $M = S^1 \times N$.
\end{itemize}
\item
For $\dim(M)=1$ the induced geodesic distance is positive for $\tfrac12< s$ and for
general $\dim(M)\geq 2$ the geodesic distance 
is positive for $1\le s$.
\end{enumerate}
\end{theorem*}

Let us now briefly outline the structure of this work. In Section~\ref{sec:frac} we review the basic definition of Sobolev-type metrics on diffeomorphism groups. For a compact Riemannian manifold $M$  (see  \ref{boundedGeometry}) we consider the Sobolev metric $H^s$ of order $s$ on the Lie algebra of vector fields and the induced right invariant metric on the diffeomorphism group $\on{Diff}(M)$. The main results regarding  geodesic distance  are contained in Section~\ref{vanishing} and~\ref{non-vanishing}. In the final section \ref{geodEq} we derive the geodesic equations for different versions of the Sobolev metric of order $s$ on $\Diff_c(\mathbb R^n)$ and discuss their relation to various well-known PDE's.

\section{Sobolev metrics $H^s$ with $s\in \mathbb R$}
\label{sec:frac}

In this section we give definitions for Sobolev norms of fractional orders and
state the properties, which we will need later to prove the vanishing geodesic distance results.

\subsection{Sobolev metrics $H^s$ on $\mathbb R^n$}
For $s>0$ the Sobolev $H^s$-norm of an $\R^n$-valued function $f$ on $\R^n$ is
defined as
\begin{equation}
\label{Hs_norm}
\| f \|_{H^s(\R^n)}^2 = \| \F\i (1+|\xi|^2)^{\frac{s}2} \F f \|_{L^2(\R^n)}^2\;,
\end{equation}
where $\F$ is the Fourier transform 
\[ \F f(\xi) = (2\pi)^{-\frac n 2} \int_{\R^n} e^{-i \langle  x,\xi\rangle} f(x)
dx \]
and $\xi$ is the independent variable in the frequency domain.
An equivalent norm is given by
\begin{equation}
\label{Hsbar_norm}
\| f \|^2_{\overline{H}^s(\R^n)} = \| f \|^2_{L^2(\R^n)} + \| |\xi|^s \F f
\|^2_{L^2(\R^n)}\;.
\end{equation}
The fact that both norms are equivalent is based on the inequality
\[ \frac{1}C \big( 1 + \sum_j |\xi_j|^s \big) \leq \big(1 + \sum_j |\xi_j|^2
\big)^{\frac{s}{2}} \leq C \big( 1 + \sum_j |\xi_j|^s \big) \]
holding for some constant $C$. For $s>1$ this says that all $\ell^s$-norms on
$\R^{n+1}$ are equivalent. 
But the inequality is true also for $0 < s < 1$, even though the expression does
not define a norm on $\R^{n+1}$.
Using any of these norms we obtain the  Sobolev spaces with non-integral $s$
$$H^s(\R^n) = \{ f \in L^2(\R^n): \| f \|_{H^s(\R^n)} < \infty \}\;.$$
These spaces are also known under the name Liouville spaces or Bessel potential
spaces. 
To make a connection with other families of function spaces, we note that the
spaces $H^s(\R^n)$ coincide with
\[ H^s(\R^n) = B^s_{22}(\R^n) = F^s_{22}(\R^n) \]
the Besov spaces $B^s_{22}(\R^n)$ and spaces of Triebel-Lizorkin type
$F^s_{22}(\R^n)$. Definitions of these spaces and a nice introduction to the
general theory of function spaces can be found in \cite[section~1]{Triebel1992}.

We will now collect some results about this family of spaces, which we will need in
the coming sections. First a result, which states that point wise multiplication
with a sufficiently smooth function is well defined.

\begin{theorem}[See theorem 4.2.2 in \cite{Triebel1992}]
\label{thm:multiplication}
Let $s>0$ and $g \in C_c^\infty(\R^n)$, a smooth function with compact support. 
Then multiplication $f \mapsto gf$ is a bounded map of $H^s(\R^n)$ into itself.
\end{theorem}
We are also allowed to compose with diffeomorphisms.

\begin{theorem}[See theorem 4.3.2 in \cite{Triebel1992}] \label{thm:composition}
Let $s>0$ and $\ph \in \Diff_c(\R^n)$ be a diffeomorphism which equals the
identity off some 
compact set. Then composition $f \mapsto f \o \ph$ is an isomorphic mapping of
$H^s(\R^n)$ onto itself.
\end{theorem}

\subsection{Sobolev metrics on  Riemannian manifolds}
\label{boundedGeometry}
Following \cite[section~7.2.1]{Triebel1992} we will now introduce the spaces
$H^s(M)$ on a compact manifold $M$. Denote by $B_\ep(x)$ the ball of radius $\ep$ with center $x$. We can choose a finite cover of $M$ by balls $B_\ep(x_\al)$ with $\ep$ sufficiently small, such that normal coordinates are defined in the ball $B_\ep(x)$, and a partition of unity $\rh_\al$, subordinated to this cover. Using this data we define the $H^s$-norm of a function $f$ on $M$ via
\begin{align*}
\| f \|_{H^s(M,g)}^2 &= \sum_{\al} \| (\rh_\al f)\o \exp_{x_\al}
\|^2_{H^s(\R^n)}\\
 &= \sum_{\al} \| \F\i (1+|\xi|^2)^{\frac{s}2} \F ((\rh_\al f)\o
\exp_{x_\al}) \|^2_{L^2(\R^n)}\;.
\end{align*}
Changing the cover or the partition of unity leads to equivalent norms, 
see \cite[theorem 7.2.3]{Triebel1992}.
For integer $s$ we get norms which are equivalent to the Sobolev norms treated
in 
\cite[chapter 2]{Eichhorn2007}. The norms depend on the choice of the Riemann
metric $g$. This 
dependence is worked out in detail in \cite{Eichhorn2007}.

For vector fields we use the trivialization of the tangent bundle 
that is induced by the coordinate charts and define the norm in each coordinate
as above. 
This leads to a (up to equivalence) well-defined $H^s$-norm on the Lie algebra
$\X_c(M)$ of compactly supported vector fields on $M$.

These definitions can be extended to manifolds of the form $M=N \times \R^n, n \geq 0$ in the obvious way and to manifolds of bounded geometry \cite{Triebel1992} using the results of~\cite{Greene78,Kordjukov91,Shubin92,Eichhorn91}.

\subsection{Sobolev metrics on $\Diff_c(M)$}
Given a norm on $\X_c(M)$ we can use the right-multiplication in the
diffeomorphism group 
$\Diff_c(M)$  to extend this norm to a right-invariant Riemannian metric on
$\Diff_c(M)$. 
In detail, given $\ph \in \Diff_c(M)$ and $X, Y \in T_\ph \Diff_c(M)$ we define
\[ G^s_\ph(X,Y) = \langle X\o \ph\i, Y\o\ph\i \rangle_{H^s(M)}\;.\]

In chapter~\ref{vanishing} and in chapter~\ref{non-vanishing} we 
are interested in questions of vanishing and non-vanishing of geodesic
distance. 
These properties are invariant under changes to equivalent inner products, since
equivalent inner products on the Lie algebra
\[ \frac{1}C \langle X, Y \rangle_1 \leq \langle X, Y \rangle_2 \leq C \langle
X, Y \rangle_1 \]
imply that the geodesic distances will be equivalent metrics
\[ \frac{1}C \on{dist}_1(\ph, \ps) \leq \on{dist}_2(\ph, \ps) \leq C
\on{dist}_1(\ph, \ps) \;.\]
Therefore the ambiguity in the definition of the $H^s$-norm is of no concern to us.

In chapter~\ref{geodEq} we will study the geodesic equations of the Sobolev
metrics on $\Diff_c(\R^n)$.
Equivalent norms may induce different geodesic equations. Therefore we will
denote the metric that is induced by the 
$H^s(\R^n)$-norm \eqref{Hs_norm} as $G^s$ and the metric that is induced  by
the 
$\overline{H}^s(\R^n)$-norm \eqref{Hsbar_norm} as $\overline{G}^s$.

\section{Vanishing geodesic distance}\label{vanishing}

\begin{theorem}[Vanishing geodesic distance]\label{vanishingthm}
The Sobolev metric of order $s$ induces vanishing geodesic distance on $\on{Diff}_c(M)$ if:
\begin{itemize}
\item
$0\leq s < \frac 12$ and $M$ a Riemannian manifold that is the product of $\R$ with a compact manifold $N$, i.e., $M = \R\times N$.
\item
$s \leq \frac 12$ and  $M$ a Riemannian manifold that is the product of $S^1$ with a compact manifold $N$, i.e., $M = S^1 \times N$.
\end{itemize}
This means that any two diffeomorphisms in the same connected component of
$\Diff_c(M)$ 
can be connected by a path of arbitrarily short $G^s$-length.
\end{theorem}
\begin{remark}\label{rem:hom}
Note that this result also implies that the geodesic distance on to the homogenous space $\Diff(S^1)/S^1$ equipped
with a homogenous metric of order $s\leq\tfrac12$ vanishes. In particular the geodesic equation for the metric of order  $s=\frac12$ on this space is the modified Constantin-Lax-Majda equation, see section~\ref{geod_eq_dim1}.
\end{remark}
We will prove this theorem by first constructing paths from the identity to some diffeomorphisms with arbitrary short length and then using the simplicity of the diffeomorphism group to show that any diffeomorphism can be connected to the identity with paths of arbitrary short length.

The restriction to $M\neq\R$ in the case $s=\frac 12$ is due to technical reasons. We believe that the result holds also in this case, however it is more difficult to construct the required vector fields in the non-compact case.

\begin{figure}[ht]
		\centering
		\def\svgwidth{0.6\textwidth}
\executeiffilenewer{images/path2.svg}{images/path2.pdf}%
{inkscape -z -D --file=images/path2.svg %
--export-pdf=images/path2.pdf --export-latex}%
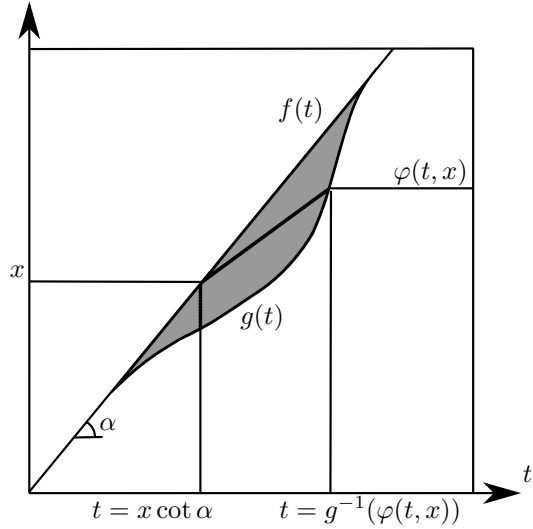%

		\caption{Sketch of the vector field $u(t,x)$. The gray area represents the support of $u$ and one integral curve of $u(\cdot,x)$ is shown.}
		\label{fig_path2}
\end{figure}

\begin{lemma}\label{lem:r}
Let $\ph \in \Diff_c(\R)$ be a diffeomorphism satisfying $\ph(x) \geq x$. 
Then for $0 \leq s < \tfrac12$ the geodesic distance between $\ph$ and
$\on{id}$ 
with respect to the $H^s$-metric in $\Diff_c(\R)$ is zero,
i.e., $\ph$ can be connected to the identity by a path of arbitrarily short
$G^s$-length.
\end{lemma}

\begin{proof}
The idea of the proof is as follows. Given the diffeomorphism $\ph$ with $\ph(x)
\geq x$ we will construct a family of paths of the form
\[ u(t,x) = \one_{[g(t), f(t)]} \star G_\ep(x) \]
such  that  its  flow  $\ph(t,x)$  satisfies $\ph(0,x) = x$ and
$\ph(T,x)=\ph(x)$. In a second step we will show, that when 
$\|f-g\|_\infty$   is   sufficiently   small,   so   is  also  the
$H^s$-length of the path $\ph(t,x)$.
Here $G_\ep (x) = \frac1{\ep}G_1(\frac{x}{\ep})$ is a smoothing kernel, where
$G_1$ is a smooth 
bump function.

So let us construct the vector field $u(t,x)$. If we could disregard continuity,
we could choose an angle $\al > \frac{\pi}{4}$ and set
\begin{align*}
f(t) &= t \tan \al \\
g^{-1}(x) &= x - (1-\cot \al) \ph^{-1}(x)\;.
\end{align*}
The flow $\ph(t,x)$ of the unsmoothed vector field $u(t,x) = \one_{[g(t),
f(t)]}(x)$ satisfies
\[ \ph(t,x) = x + \int_0^t u(s, \ph(s,x)) ds\;.\]
In this case we can write down the explicit solution, which is given by
\[ \ph(t,x) = 
\begin{cases}
x,& t < \cot \al \\
t + (1-\cot \al) x,& x \cot \al \leq t \leq \ph(x) - (1-\cot \al) x \\
\ph(x),& t > \ph(x) - (1-\cot \al) x
\end{cases}
\]
and we see that it satisfies the boundary conditions. We also have the relation
\[ f^{-1}(x) - g^{-1}(x) = -(1-\cot \al)(x - \ph^{-1}(x)) \]
which implies that by choosing $\al$ sufficiently close to $\frac\pi 4$ we can
make $\| f-g\|_\infty$ as small as necessary. 
By replacing $u$ with the smoothed vector field $\one_{[g(t), f(t)]} \star
G_\ep(x)$ we change the endpoint of the flow. 
However, by changing $g$ suitably we can regain control of the endpoint. 
The necessary changes will be of order $\ep$ and hence we don't loose control 
over the difference $\| f-g\|_\infty$, which will be necessary later on.

Now we compute the norm of this vector field. Let $u(t,x)$ have the form $u(t,x)
= \one_{[g(t), 
f(t)]} \star G_\ep(x)$, where $f(t)$ and $g(t)$ are smooth functions which
coincide off a bounded interval. 
To compute the $H^s$-norm of $u$, we first need to compute its Fourier transform
\begin{align*}
\F \one_{[g(t), f(t)]}(\xi) &= \frac{1}{\sqrt{2\pi}} \int_{g(t)}^{f(t)} e^{i\xi
x} \ud x = \frac{1}{\sqrt{2\pi}} \left. \frac{e^{i \xi x}}{i \xi} \right
\vert_{x=g(t)}^{x=f(t)} = \frac{1}{\sqrt{2\pi}} \frac{e^{i \xi f(t)} - e^{i\xi
g(t)}}{i \xi} \\
&= \frac{1}{\sqrt{2\pi}} \frac{1}{\xi} e^{i \xi \frac{f(t)+g(t)}{2}} \frac{e^{i
\xi \frac{f(t)-g(t)}{2}} - e^{-i\xi \frac{f(t)-g(t)}{2}}}{i \xi} \\
&= \frac{2}{\sqrt{2\pi}} \frac{1}{\xi} e^{i \xi \frac{f(t)+g(t)}{2}} \sin(\xi
\frac{f(t)-g(t)}{2})\;.
\end{align*}
Setting $a = \frac{f(t)-g(t)}{2}$ we can now compute the norm
\begin{align*}
\| \xi^s \F u \|^2_{L^2} &= \| \xi^s \F \one_{[g(t), f(t)]} \F G_\ep \|^2_{L^2}
\\
&= \int_\R \frac{2}{\pi} \frac{1}{|\xi|^{2-2s}} \sin^2(a\xi) (\F G_1(\ep\xi))^2
d\xi \\
&\leq \frac{2}{\pi} \| \F G_1 \|^2_{\infty} \int_\R \frac{\sin^2(a
\xi)}{|\xi|^{2-2s}} d\xi \\
&\leq \frac{2}{\pi} \| \F G_1 \|^2_{\infty} \int_\R a^{1-2s}
\frac{\sin^2(\xi)}{|\xi|^{2-2s}} d\xi\;.
\end{align*}
We get $\| u \|_{L^2}^2$ by setting $s=0$ in the above calculation. 
We see that for $s < \frac{1}{2}$ the $H^s$-norm of $u(t,.)$ is bounded by
\[ \| u(t,.) \|_{H^s}^2 \leq C_1 |f(t) - g(t)| + C_2 |f(t) - g(t)|^{1-2s}\;.\]
Now, putting everything together
\begin{align*}
\on{Len}(\ph)^2 &= \left(\int_0^T \| u(t,.)\|_{H^s} dt\right)^2 \leq T \int_0^T
\| u(t,.)\|^2_{H^s} dt \\
&\leq T^2 \left(C_1 \| f - g \|_\infty + C_2 \| f - g \|_\infty^{1-2s}\right)\;.
\end{align*}
Since  the  geodesic  length is defined as the infimum over all
paths  and  since  we have shown in the first part of the proof
that  by  choosing  the angle $\alpha$ and the smoothing factor
$\ep$,  we  can control the norm $\|f-g\|_\infty$, the proof is
complete.
\end{proof}

In the case $s=\frac 12$ we will work on the circle $S^1$. 
Using a construction similar to that in Lemma \ref{lem:r} 
we will construct arbitrary short paths from the identity to the shift $\ph(x) = x+1$. 
The following lemma supplies us with functions that have small $H^{1/2}$-norm and large $L^\infty$-norm at the same time. 
We cannot use step functions as in the above proof, because they are not in $H^{1/2}(S^1)$.

\begin{lemma}\label{lem:unbounded}
Let $\psi(x)$ be a  non-negative, compactly supported $C^\infty$-function on
$\R$ and $(b_j)_{j=0}^{\infty}$ a non-increasing sequence of non-negative
numbers, with 
$\sum b_j^2<\infty$.    Then the $H^{1/2}$-norm of the function $f(x):= \sum b_j
\psi(2^j x)$ is bounded by
$$\|f\|^2_{H^{1/2}}\leq C\sum_{j=0}^\infty b_j^2\;.$$
\end{lemma}
\begin{proof}
This result is shown in step 4 of the proof of theorem 13.2
in~\cite{Triebel2001}.
\end{proof} 

The main difference between Lemma \ref{lem:r} and the following Lemma is that on $S^1$ we don't have to worry about the diffeomorphisms having compact support. On $\R$ the diffeomorphism $\ph(x) = x+1$ is not element of $\on{Diff}_c(\R)$ and we would have to replace it by $\ph(x) = x + c(x)$, where $c(x)$ is some function with compact support. This makes working on the circle much easier.

\begin{lemma}\label{lem:r2}
Let $\ph \in \Diff(S^1)$ be the shift by $1$, i.e. $\ph(x)=x+1$. 
Then  the geodesic distance between $\ph$ and $\on{id}$ 
with respect to the $H^{1/2}$-metric in $\Diff(S^1)$ is zero.
\end{lemma}
\begin{proof}
We will prove the lemma by constructing a sequence of vector fields with arbitrary small 
$H^{1/2}$-norms, whose flows at time $t=T_{\on{end}}$ will be $\ph(T_{\on{end}},x)=x+1$. 
First we apply Lemma \ref{lem:unbounded} with $b_j = \frac 1N$ for $j=0,\ldots,N-1$ and 0 otherwise. 
By doing so we obtain $\| f \|^2_{H^{1/2}} \leq C\frac 1N$, while $\|f\|_\infty = 1$. For the basic function $\ps(x)$ we choose $\ps(x)=e^{\frac{1}{1-|x|^2}}$. Note that $\on{supp}(\ps) \subseteq [-1,1]$ and that $\ps$ is concave in a neighborhood around 0. Since each $f$ is a finite sum of $\ps(2^j x)$, these properties hold also for $f$.

We define the vector field $$u(t,x)=\la f(t-x) \quad\text{with}\quad 0 \le \la < 1$$ 
for $t \in [0, T_{\on{end}}]$, where $T_{\on{end}}$ will be specified later. 
The energy of this path is bounded by
\[ E(u) = \int_0^{T_{\on{end}}} \| u(t,.)\|^2_{H^{1/2}} \ud t \leq C T_{\on{end}} \frac 1N \]
and hence can be made as small as necessary. It remains to show that the flow of this vector field 
at time $t=T_{\on{end}}$ is indeed $\ph(T_{\on{end}},x) = x+1$. 

We do this in several steps. First we consider this vector field defined on all of $\R$ with time 
going from $-\infty$ to $\infty$. The initial condition for the flow is $\ph(-\infty,x) = x$. Since 
$u(t,x)$ has compact support in $x$, this doesn't cause any analytical problems. As long as $\la < 
1$ each integral curve of $u$ will leave the support of $u$ after finite time. 
Therefore we can consider $\ph(\infty,x)$ to be the endpoint of the flow. 
Next we will establish that $\ph(\infty,x) = x + S$ is a uniform shift, that is independent of $x$. 
Then we show that by appropriately choosing $\la$ we can control the amount of shifting, in 
particular we can always obtain $S=1$. 
Then we find bounds for the time each integral curve spends in the support of $u$. 
By showing that this time is only dependent on $S$, but not on the specific form of $f$ or $\ps$, we will know that $T_{\on{end}}$ doesn't grow larger as we let $N\to\infty$. In the last step we go back to the circle, define $T_{\on{end}}$, start the flow at time $t=0$ and show that the resulting flow is a shift by 1 at time $T_{\on{end}}$. This will conclude the proof.

The flow $\ph(t,x)$ of $u(t,x)$ is given by the equation
\[ \p_t \ph(t,x) = u(t, \ph(t,x)) \]
with the initial condition $\ph(-\infty, x) = x$. Define the function $a_x(t) = t - \ph(t,x)$. Because $\p_t a_x(t) = 1 - \la f(a_x(t)) >0$ the function $a_x(.)$ is a diffeomorphism in $t$ for each fixed $x$. Since $\on{supp}(f) \subseteq [-1,1]$, we have $\p_t \ph(t,x) \neq 0$ only for $t \in [a_x\i(-1), a_x\i(1)]$. Let us define $T_{\on{shift}} = a_x\i(1) - a_x\i(-1)$ to be the time necessary for the flow to pass through the vector field $u$.

{\bf Claim A.} {\em $T_{\on{shift}}$ is independent of $x$.} \\
This follows from the following symmetries of the flow $\ph(t,x)$ and the map $a_x(t)$. We have
\[ \ph(t,x) = \ph(t-(x-y),y) + x-y \]
and
\[ a_x(t) = a_y(t-(x-y))\;. \]
To prove the first identity assume $x>y$ and note that at time $t_0 = y-1$ we have
\[ \ph(y-1,x)=x = x + \ph(y-1-(x-y),y)-y \]
since at time $y-1-(x-y)$ the flow $\ph(t,y)$ still equals $y$. Now differentiate to see that both functions satisfy the same ODE. The second identity is an immediate consequence of the first one. To prove the claim that $T_{\on{shift}}$ is independent of $x$, we will show that
\[ \p_x a_x\i(t) = 1\;. \]
Start with $a_x(a_x\i(t)) = t$, use the symmetry relation $a_0(a_x\i(t)-x) = t$ and differentiate with respect to $x$ to obtain
\[ \p_t a_0(a_x\i(t)-x) (\p_xa_x\i(t)-1) = 0\;, \]
which concludes the proof of the claim.

For each $x$, the flow of the vector field performs a shift $[a_x\i(-1), a_x\i(1)]$ given by
\[ \ph(\infty,x) = x + \int_{a_x\i(-1)}^{a_x\i(1)} \la f(t-\ph(t,x)) \ud t = x + \int_{-1}^{1} \frac{\la f(t)}{1-\la f(t)} \ud t\;. \]
Define 
$$
I(\la) = \int_{-1}^{1} \frac{\la f(t)}{1-\la f(t)} \ud t
$$ 
to be the amount of shifting that is taking place as a function of $\la$. 
Note that $I(\la)$ is smooth in $\la\in[0,1)$ with $I(0)=0$ and $I'(0)>0$.
We claim that we can always choose $\la$ close enough to 1, to obtain any shift necessary.

{\bf Claim B.} {\em $I:[0,1)\to \mathbb R_{\ge 0}$ is a diffeomorphism}  \\
Obviously, $\p_\la I(\la)>0$ and $I(0)=0$. It remains to show that 
$\lim_{\la\to 1}I(\la) = \infty\;$.
Each $f$ that we choose in our construction is concave in some small neighborhood around 0. So choose $a>0$, such that for $t \in [0,a]$ we have $f(t)\geq \frac 12$ and $f(t) \geq 1-ct$ for some constant $c$. Then we can estimate the integral by
\begin{align*}
I(\la) &\geq \int_0^a \frac{\la f(t)}{1-\la f(t)} \ud t \geq \frac \la 2 \int_0^a \frac{1}{1-\la + \la c t} \ud t \\
&\geq \frac \la 2 \log(1-\la+\la c t)\big|_{t=0}^{t=a} \\
&\geq \frac \la 2 \log\left(1 + \frac{\la c a}{1-\la} \right)\;.
\end{align*}
From this we can see that $I(\la)$ grows towards infinity.

Note that a similar calculation shows that $\p_\la I(\la)>a$ for some $a\geq 0$.

\begin{figure}[ht]
		\centering
		\def\svgwidth{0.8\textwidth}
\executeiffilenewer{images/path.svg}{images/path.pdf}%
{inkscape -z -D --file=images/path.svg %
--export-pdf=images/path.pdf --export-latex}%
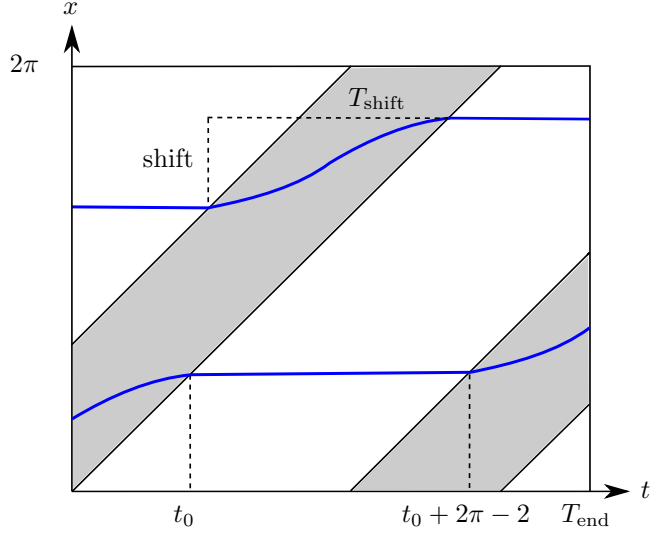%

		\caption{Sketch of the vector field $u(t,x)$. The gray area represents the support of $u$ and the blue curves are integral curves of $u(\cdot,x)$.}
		\label{fig:sketch}
\end{figure}

{\bf Claim C.} {\em We have control over the time necessary to induce a shift $I(\la)$, via  $T_{\on{shift}}= 2+I(\la)\;$. } \\
We had above $T_{\on{shift}} = a_x\i(1) - a_x\i(-1)$ and also $\p_t a_x\i(t) = \frac{1}{1-\la f(t)}\;$. Hence
\[ T_{\on{shift}} = \int_{-1}^1 \frac{1}{1-\la f(t)} \ud t = \int_{-1}^1 1 + \frac{\la f(t)}{1-\la f(t)} \ud t = 2 + I(\la)\;. \]
This concludes the proof of the claim.

Now we choose $\la$ such that $I(\la)=1$ and we define the flow $\ph(t,x)$ on the circle with period $2\pi$ for time from 0 to $T_{\on{end}}=T_{\on{shift}} + 2\pi - 2\;$.

{\bf Claim D.} {\em The endpoint of the resulting flow is a constant shift $$\ph(T_{\on{end}},x) = 
x+1\;.$$} 
We have to consider two cases. First take a point $x$, such that $x \notin \on{supp}(f)\;$. W.l.o.g. we assume that $\on{supp}(f)$ is the interval $[0, 2]$ starting at 0, which implies that in this case $x > 2\;$. Then $x$ meets the vector field at time $t_0=x-2$ and leaves it again at time $t_1 = x-2+T_{\on{shift}}$ after being shifted by one. It would meet the vector field again at time $t_2 = x-2+T_{\on{shift}} + 2\pi - 2\;$, but because $x-2 >0$ we have $t_2 > T_{\on{end}}$ and hence the point doesn't meet the vector field again.

Now take a point $0\leq x \leq 2\;$. This point starts in the vector field, leaves it at some time 
$t_0 < T_{\on{shift}}$ and meets the vector field again at time $t_0 + 2\pi - 2\;$. It then stays 
in the vector field until the $T_{\on{end}}$ since $T_{\on{end}} < t_0 + 2\pi - 2 + 
T_{\on{shift}}\;$. Altogether the point has spent time $t_0 + T_{\on{end}} - (t_0 + 2\pi -2) = 
T_{\on{shift}}$ in the vector field, so it has also been shifted by $1$. This concludes the proof.
\end{proof}

\begin{proof}[Proof of theorem \ref{vanishingthm}]
To prove vanishing of the geodesic distance we will follow the idea of
\cite{Michor102}, where it was proven that the geodesic distance vanishes
on $\on{Diff}_c(M)$ for the right-invariant $L^2$-metric. To be precise,
we only consider the connected component $\on{Diff}_0(M)$ of $\on{Id}$,
i.e. those diffeomorphisms of $\on{Diff}_c(M)$, for which there exist at
least one path, joining them to the identity. Let us denote by
$\on{Diff}_c(M)^{L=0}$ the set of all diffeomorphisms $\ph$ that
can be reached from the identity by curves of arbitrarily short length,
i.e., for each $\ep>0$ there 
exists a curve from the identity to $\ph$ with length smaller than $\ep$.

In the following we will show that $\on{Diff}_c(M)^{L=0}$ is a non-trivial
normal
subgroup 
of $\on{Diff}_c(M)$ (and $\on{Diff}_0(M)$). It was shown
in~\cite{Epstein1970,Thurston1974,Mather1974,Mather1975} that
$\on{Diff}_0(M)$ is a simple group, i.e. only the identity and the whole group are normal subgroups of $\on{Diff}_0(M)$. From this it follows that $\on{Diff}_c(M)^{L=0}$ is the whole
connected component $\on{Diff}_0(M)$. In other words, every diffeomorphism
that can be connected to the identity, can be connected via a path of
arbitrary short length.

{\bf Claim A.} \emph{$\Diff_0(M)^{L=0}$ is a normal subgroup of $\Diff_0(M)$.}
\\Given a diffeomorphism $\ps \in \Diff_0(M)$, we can choose a partition of
unity $\ta_j$ such that 
normal coordinates centered at $x_j\in M$ are defined on $\on{supp}(\ta_j)$ and
such that normal coordinates
centered at $\ps(x_j)$ are defined on $\ps(\on{supp}(\ta_j))$.
Then we can define $\ps_j = \exp_{\ps(x_j)}\i \o \ps \o \exp_{x_j}$.
For $\ph_1 \in \Diff_0(M)^{L=0}$ we choose a curve 
$t \mapsto \ph(t,\cdot)$ from the identity to $\ph_1$ with length less than
$\ep$. 
Let  $u=\ph_t \o \ph\i$. Then
\begin{align*}
\on{Len}(\ps\i \o \ph \o \ps)&\leq
C_1(\ta)\int_0^1 \| (T\ps\i \o \ph_t \o \ps) \o (\ps\i \o \ph \o \ps)\i
\|_{H^s(M,\ta)} dt \\&=
C_1(\ta)\int_0^1 \| T\ps\i \o u \o \ps \|_{H^s(M,\ta)} dt \\&=
C_1(\ta)\int_0^1 \sqrt{\sum_{j} \| \exp_{x_j}^*(\ta_j.T\ps\i \o u \o
\ps) \|^2_{H^s(\R^n)}} dt \\&=
C_1(\ta)\int_0^1 \sqrt{\sum_{j} \| T\ps_j\i . (\exp_{\ps(x_j)}^*
(\ta_j\o\ps\i.u)) \o \ps_j \|^2_{H^s(\R^n)}} dt \\&\leq
C_2(\ps,\ta) \int_0^1 \sqrt{\sum_{j} \| (\exp_{\ps(x_j)}^*
(\ta_j\o\ps\i.u)) \|^2_{H^s(\R^n)}} dt 
\\&=C_2(\ps,\ta) \int_0^1  \|u \|_{H^s(M,\ta\circ\psi\i)} dt \leq C_3(\ps,\ta)
\on{Len}(\ph)\;.
\end{align*}
Here we used that all partitions of unity $\ta$ induce equivalent norms
$H^s(M,\ta)$
and that for $h\in C^\infty(M)$ and $\ps\in \Diff_c(M)$ point wise multiplication
$f\mapsto h.f$ 
and composition $f\mapsto f\o\ps$ are bounded linear operators on
$H^s(M)$, as noted in theorems~\ref{thm:multiplication} and
\ref{thm:composition}.

{\bf Claim B.} \emph{$\Diff_0(M)^{L=0}$ is a nontrivial subgroup of
$\Diff_0(M)\;$.}\\
For a one-dimensional manifold $M$ the non-triviality  of $\Diff_0(M)^{L=0}$ under appropriate conditions on $s$ is shown in lemmma~\ref{lem:r} and lemma~\ref{lem:r2}. 
The higher dimensional cases treated here are quite obvious, since we can endow $M=S^1\times N$ or $M=\R\times N$ with a product metric and use the paths obtained in the corresponding one-dimensional case. 
To spell this out let $\ph_{\ep}(t,x)$ be a family of paths on $\Diff_c(\R)$ or $\Diff(S^1)$ indexed by $\ep$ such that each path $\ph_\ep$ connects the identity to some diffeomorphism $\ph_\ep(T,x)=\ph(x)$ and has $H^s$-length smaller than $\ep$. Then 
$\psi_\ep(t,(x,y)) = (\ph_\ep(t,x),y)$ defines a family of paths on $\Diff_c(M)$ connecting $\Id_M$ to $(\ph,\on{Id}_N)$ with
arbitrarily short $H^s$-length.

\end{proof}

\section{Positive Geodesic Distance}\label{non-vanishing}

\begin{theorem}[Positive geodesic distance]\label{positive}
For $\on{dim}(M)=1$ the Sobolev-norm of order $s$ induces positive geodesic
distance on $\Diff_c(M)$ if
$s>\frac{1}{2}$. For $\on{dim}(M)\geq 2$ it induces positive geodesic distance
if $s\geq 1$.
\end{theorem}

\begin{proof}
By the definition of the Sobolev metric it suffices to show the result for
$\Diff_c(\R^n)$.

For the case $n=1$ let $\ph_0,\ph_1\in\Diff_c(\R)$ with $\ph_0(x) \neq \ph_1(x)$
for some $x\in\R$. For any path $\ph(t,\cdot)$,
with $\ph(0,\cdot)=\ph_0$ and $\ph(1,\cdot)=\ph_1$ we have
\begin{align*}
0&\neq|\ph_1(x)-\ph_0(x)|=\left|\int_0^1 \ph_t(t,x) dt \right|=\left|\int_0^1
u(t,\ph(t,x)) dt \right|\\&
\leq \int_0^1 \left|u(t,\ph(t,x))\right| dt \leq \int_0^1 \|u(t,\cdot)\|_\infty
dt 
\leq \int_0^1 \|u(t,\cdot)\|_{C^{0,s-1/2}}dt\\&\leq \int_0^1
\|u(t,\cdot)\|_{H^{s}}dt\;.
\end{align*}
In the last step, we used the Sobolev embedding theorem, see \cite{Triebel1992}
for example. The case 
$\dim(M)\geq 2$ follows from  \cite[Theorem~5.7]{Michor102}.
\end{proof}

\section{The Geodesic Equation on $\Diff_c(\R^n)$}\label{geodEq}

In the upcoming parts we want to calculate the geodesic equation for the  two
equivalent Sobolev norms
on $\Diff_c(\R^n)$. 

\subsection{The general setting}\label{setting}
According to \cite{Arnold1966}, see \cite[section 3]{Michor109} for a
presentation directly 
applicable here, we have:
For any right invariant metric $G$ on a regular infinite dimensional Lie group, 
the  geodesic equation reads as 
$$
u_t=-\ad(u)^\top u\;.
$$
Here $\on{ad}(u)^\top$ denotes the adjoint of the adjoint representation $\ad$, which is given by
$\langle \ad(v)^{\top} u, w \rangle_G:=\langle u, \ad_v w \rangle_{G}\;$. 
Note that for $\Diff_c(\R^n)$ we have $\ad_vw = -[v,w]$ for $v,w\in \X_c(\mathbb
R^n)\;$.
The sectional curvature (for orthonormal $u,v$) at the identity is then given by
the formula
$$
\langle R(u,v)v, u \rangle_G=\frac14\|\be(u)v-\be(v)u-\on{ad}(u)v\|_{G}^2
+\langle [\be(u),\be(v)]u,v \rangle_G
$$
where $\be(u)v:= \on{ad}(u)^\top v +\on{ad}(u)v\;$.
This last expression is from \cite[section 2.6]{MMM11}.
The Jacobi equation for a right trivialized Jacobi field $y$ along a geodesic
with right trivialized 
velocity field $u$ (which satisfies the geodesic equation) is derived in 
\cite[3.4 and 3.5]{Michor69} as:
$$
y_{tt}= [\ad(y)^\top+\ad(y),\ad(u)^\top]u   
     - \ad(u)^\top y_t -\ad(y_t)^\top u + \ad(u)y_t\,.
$$ 
This will allow us to write down the curvature and the Jacobi equations for all
metrics that we 
will treat below. Since this leads to complicated formulas we will not spell
this out.

\begin{theorem}\label{th:ge}
Let $A: \X_c(\R^n) \to \X_c(\R^n)$ be an elliptic, scalar (pseudo)-differential
operator that  is positive and 
self-adjoint with respect to the $L^2$-metric. Then $A$ induces a metric on
$\Diff_c(\R^n)$ in the following way:
$$G_{\ph}^A(X,Y):=\langle A (X\o\ph\i),Y\o\ph\i \rangle_{L^2(\R^n)}\;.$$
The geodesic equation with respect to the $G^A$-metric is then given by
\[ A u^k_t= - \sum_{i=1}^n\left( A u^i(\p_ku^i)+(A(\p_i u^k).u^i+A u^k.(\p_i
u^i))\right).\]
Equivalently it can be written in terms of the  momentum $m=A u$:
\[ m^k_t =- \sum_{i=1}^n\left( m^i(\p_ku^i)+((\p_i m^k).u^i+m^k.(\p_i
u^i))\right),\quad u^k = A\i m^k .\]
\end{theorem}

\begin{proof}
For $u,v,w\in\X_c(\R^n)$ we calculate
\begin{align*}
&\langle u, -[v,w] \rangle_{G^A}=\int_{\R^n} \langle  A
u,-[v,w]\rangle_{\R^n}\ud x\\ &\qquad
=\int_{\R^n}\sum_{k=1}^n
Au^k\sum_{i=1}^n\left((\p_iv^k)w^i-v^i(\p_iw^k)\right)\ud x\\ &\qquad
=\int_{\R^n}\sum_{k=1}^n\sum_{i=1}^n
Au^k(\p_iv^k)w^i+\p_i(Au^k.v^i)\p_iw^k\ud x\\ &\qquad
=\int_{\R^n}\sum_{i=1}^n\sum_{k=1}^n
Au^i(\p_kv^i)w^k+\sum_{k=1}^n\sum_{i=1}^n(A(\p_i u^k).v^i+A u^k.(\p_i
v^i))w^k\ud x\\ &\qquad
=\int_{\R^n} \sum_{k=1}^n w^k A\left(\ad(v)^{\top} u\right)^k \ud x=\langle
\ad(v)^{\top} u, w\rangle_{G^A}\;.\\
&\ad(v)^{\top} u = \sum_{i=1}^n A\i\left( A u^i(\p_ku^i)+(A(\p_i
u^k).u^i+A u^k.(\p_i u^i))\right)\;.\qedhere
\end{align*}
\end{proof}

\begin{theorem}[Geodesic equation for the Sobolev metric $G^s$]
The operator $$A_s:H^{k+2s}(\R^n) \to H^{k}(\R^n),\quad u(x) \mapsto\left(\F\i
(1+|\xi|^2)^{s}\F u\right) (x)$$  induces the Sobolev metric $G^s$ of order $s$
on $\Diff_c(\R^n)$. 
The geodesic equation for this metric reads as:
\begin{align*}
m^k_t &=- \sum_{i=1}^n\left( m^i(\p_ku^i)+((\p_i m^k).u^i+m^k.(\p_i
u^i))\right), \\
u^k &=\begin{cases} 
(2\pi)^{\frac{n}{2}}\frac{2^{1-s} |\cdot|^{s-\frac{n}{2}}}{\Ga(s)}
K_{s-\frac{n}{2}}(|\cdot|)\star m^k,\quad s>\frac{n-1}{4}\\
(2\pi)^{\frac
n2}|\cdot|^{1-\frac{n}{2}}\int_0^{\infty}J_{\frac{n}{2}-1}(r.|\cdot|)\frac{r^{
\frac{n}{2}}}{(1+r^2)^{s}}\ud r \star m^k,\quad s\leq\frac{n-1}{4}\;.
\end{cases}  
\end{align*}
Here $J_{n/2-1}$ denotes the Bessel function of the first kind, which is given
by
$$J_\alpha(r) =
   \frac{1}{\pi} \int_0^\pi \cos(\alpha t- r \sin t)\,d t
 - \frac{\sin(\alpha\pi)}{\pi} \int_0^\infty
          e^{-r \sinh(t) - \alpha t} \, dt\;,$$
and  $K_{s-\frac n2}$ denotes the modified Bessel function of second kind, which
is given by
\[ K_\nu(r) = \frac{\Ga(\nu+\frac 12)(2r)^\nu}{\sqrt \pi} \int_0^\infty
\frac{\cos t}{(t^2+r^2)^{\nu+\frac 12}} \ud t\;.\]
\end{theorem}

\begin{proof}
The operator $A_su=\F\i (1+|\xi|^2)^{s}\F u $ is an elliptic, scalar
(pseudo)-differen\-tial operator that  is positive and 
self-adjoint with respect to the $L^2$-metric. In particular it is a linear
isomorphism from  $H^{k+2s}(\R^n) \to H^{k}(\R^n)$.
By theorem \ref{th:ge} 
it  remains to calculate the operator $A_s\i$.
This can be done  as follows:
\begin{align*}
A\i_sm(x) &= \F\i (1+|\xi|^2)^{-s} \F m(x)= (2\pi)^{\frac{n}{2}} \F\i
\left((1+|\xi|^2)^{-s}\right) \star m(x)\;.
\end{align*}
Let $f$ be a radial symmetric function  on $\R^n$, i.e., $f(\xi)=f(|\xi|)$. Then
we have 
$$\F\i(f)(x)=|x|^{1-\frac{n}{2}}\int_0^{\infty}J_{\frac{n}{2}-1}(r.|x|).f(r).r^{
\frac{n}{2}}dr\;.$$ 
For the function $f(\xi)=(1+|\xi|^2)^{-s}$ this yields:
\begin{align*}
\F\i \left((1+|\xi|^2)^{-s}\right)&
=|x|^{1-\frac{n}{2}}\int_0^{\infty}J_{\frac{n}{2}-1}(r.|x|)\frac{r^{\frac{n}{2}}
}{(1+r^2)^{s}}\ud r\;.
\end{align*}
See for example the books \cite{Stein2004,Stein1971} for more details about
Fourier transformation of radial symmetric functions.
For $s>\frac{n-1}{4}$ the last integral converges and we have
\[ \F\i \left((1+|\xi|^2)^{-s}\right)=\frac{2^{1-s} |x|^{s-\frac{n}{2}}}{\Ga(s)}
K_{s-\frac{n}{2}}(|x|)\;.\qedhere\]
\end{proof}
An immediate consequence of the above analysis is the geodesic equation for the
equivalent Sobolev-metric $\overline{G}^s$.
\begin{theorem}[Geodesic equation for the Sobolev metric $\overline{G}^s$]
The operator $$\overline{A}_s:H^{k+2s}(\R^n) \to H^{k}(\R^n),\quad u(x)
\mapsto\left(\F\i (1+|\xi|^{2s})
\F u\right) (x)$$  
induces the Sobolev metric $\overline{G}^s$ of order $s$ on $\Diff_c(\R^n)$. 
The geodesic equation for this metric reads as:
\begin{align*}
m^k_t &=- \sum_{i=1}^n\left( m^i(\p_ku^i)+((\p_i m^k).u^i+m^k.(\p_i
u^i))\right)\;, \\
u^k &=(2\pi)^{\frac
n2}|\cdot|^{1-\frac{n}{2}}\int_0^{\infty}J_{\frac{n}{2}-1}(r.|\cdot|)\frac{r^{
\frac{n}{2}}}{(1+r^{2s})}\ud r\star m^k\;.
\end{align*}
\end{theorem}

\subsection{The Geodesic Equation in dimension one}\label{geod_eq_dim1}

For the $\overline G^{s}$-metric on $\Diff_c(\R)$ or $\Diff(S^1)$
 the above expression for the geodesic equation simplifies to
\[ m_t = -2u_xm - um_x,\quad u = (2\pi)^{\frac
12}\int_0^{\infty}J_{-\frac{1}{2}}(r.|\cdot|)\frac{r^{\frac{1}{2}}}{(1+r^{2s})}
\ud r\star m\;. \]
For $s=k\in \N$ we can rewrite this equation as:
\[m_t = -2u_xm - um_x, \quad m=u+\p^{2 k}_x u\;,\]
where $m(t,x)$ is the momentum corresponding to the velocity $u(t,x)$. 
In the case $s=0$ this becomes the inviscid Burger equation
$$u_t=-3u_x u\;,$$
for $s=1$ it is the Camassa Holm equation
$$u_t- u_{xxt} + 3 u u_x = 2 u_x u_{xx} + u u_{xxx}\;,$$
and for $s>1$ they are related to the higher order Camassa Holm equations, see
\cite{Constantin2003}.

To study the Sobolev metric $\overline{G}^{s}$, for $s=k+\frac12$ we introduce
the Hilbert transform, which is given by
\[\F(\mathcal H f)(\xi)=-i\on{sgn}(\xi)\F f(\xi)\;. \]
Using this we can write the metric $\overline{G}^{k+\tfrac12}$  in the form
\begin{align*}
\overline{G}^{k+\tfrac12}(u,v)=\int_\R (u+\mathcal H \p^{2k+1}_xu)v \ud
x=\int_\R (u+\p^{2k+1}_x(\mathcal H u))v \ud x\;.
\end{align*}
The geodesic equation is then given by
\[ m_t = -2u_xm - um_x,\quad m=u + \mathcal H \p^{2k+1}_xu\;.\]
If we pass to the homogenous space $\Diff(S^1)/S^1$ and consider the homogenous
Sobolev metric 
of order one half the geodesic equation reduces to
\[ m_t = -2u_xm - um_x,\quad m= \mathcal H u_x\;,\]
which is  the modified Constantin-Lax-Majda equation, see \cite{Wunsch2010a}. 
If we consider the homogenous Sobolev metric of order one the resulting geodesic
equation 
\[u_{xxt} = -2u_xu_{xx} - uu_{xxx}\]
is the Hunter-Saxton equation. 

\section{Conclusions.}

In this paper we have provided a partial answer to the problem of vanishing geodesic distance for Sobolev-type metrics on diffeomorphism groups. Some cases remain open.

{\bf Conjecture.} For $M=\R$ and $s=\frac 12$ the geodesic distance vanishes.

We believe that a similar construction as in Lemma \ref{lem:r2} can be used to
construct paths of arbitrary short length. However, since the diffeomorphisms need
to have compact support, we need to adapt the construction to reach a
diffeomorphism of the form $\ph(x)=x+c(x)$ as in Lemma \ref{lem:r}. The difficulty
lies in constructing a vector field, whose flow at time $t=1$ we can control.

The more interesting question is about the behaviour for arbitrary Riemannian manifolds of
bounded geometry and higher dimension:

{\bf Conjecture.} 
The geodesic distance vanishes on the space $\on{Diff}_c(M)$ for $0 \leq s
\leq \frac 12$, for $M$ an arbitrary manifold of bounded geometry.

For $0 \leq s \leq \frac 12$ we believe that it is
possible to adapt the method of \cite{Michor102} to show that there exists a diffeomorphism that can be
connected to the identity by paths of arbitrary short length. 
The simplicity of the diffeomorphism group then concludes the proof.

{\bf Conjecture.} The geodesic distance is positive for $s>\frac 12$. 

The proof of Theorem~\ref{positive} establishing positivity of the geodesic distance for $\dim(M)=1$ and $s>\frac12$ 
is based on the Sobolev embedding theorem, which holds
for $s$ greater than the critical index $\on{dim}(M)/2$. The argument generalizes to higher dimensions, 
but in $\dim(M)\geq2$ the result \cite[Theorem~5.7]{Michor102} is stronger. Namely, it is shown that 
the geodesic distance is positive for $s\geq 1$ in all dimensions. To prove the conjecture 
it remains to improve the bound $s\geq1$ to $s>\frac12$ for $\dim(M)\geq 2$.

The proof of Theorem~\ref{vanishingthm} provides a hint that this can be done and that the 
bound $s>\frac12$ is optimal. 
The idea of the proof in dimension one is to compress the space to a point and to move
this point around. This results in a path of diffeomorphisms with short $H^s$-length 
because there are functions with small $H^s$-norm that are nevertheless large at some point. 
The obvious generalization to higher dimensions
is to compress the space to a set of codimension one and to move this set around.
Again, this should result in a path of short $H^s$-length when there are functions with small $H^s$-norm 
that are large at a set of codimension one. This is the case exactly for $s\leq \frac 12$.

Another interesting question is whether our results carry over to the Virasoro-Bott group.
In \cite{Michor122} it was shown that the right invariant $L^2$-metric on the
Virasoro-Bott group has vanishing geodesic distance. The key to the proof in \cite{Michor122} is 
to control the central cocycle along a curve of diffeomorphisms. This is done by expressing the cocycle
in terms of the diffeomorphism and its derivatives. In contrast to this, 
all derivatives in the constructions of the present work are left-trivialized. 
We believe that this is not a serious obstacle and that the results of this paper can be extended to Sobolev metrics of fractional order 
on the Virasoro-Bott group.

\bibliographystyle{plain}

\end{document}